\documentclass[11pt]{article}
\usepackage{amsmath,amssymb}

\newtheorem{propo}{{\bf Proposition}}[section]
\newtheorem{coro}[propo]{{\bf Corollary}}
\newtheorem{lemma}[propo]{{\bf Lemma}} 
\newtheorem{theor}[propo]{{\bf Theorem}} 
\newtheorem{ex}{{\sc Example}}[section]

\newenvironment{proof}{{\bf Proof.}}{$\Box$}

\begin{document}

\vspace*{1.0in}

\begin{center} ON CONJUGACY OF MAXIMAL SUBALGEBRAS OF SOLVABLE LIE ALGEBRAS  
\end{center}
\bigskip

\begin{center} DAVID A. TOWERS 
\end{center}
\bigskip

\begin{center} Department of Mathematics and Statistics

Lancaster University

Lancaster LA1 4YF

England

d.towers@lancaster.ac.uk 
\end{center}
\bigskip

\begin{abstract} The purpose of this paper is to consider when two maximal subalgebras of a finite-dimensional solvable Lie algebra $L$ are conjugate, and to investigate their intersection.
\end{abstract}

Throughout $L$ will denote a finite-dimensional solvable Lie algebra over a field $F$. Let $x \in L$ and let ad\,$x$ be the corresponding inner derivation of $L$. If $F$ has characteristic zero suppose that (ad\,$x)^n = 0$ for some $n$; if $F$ has characteristic $p$ suppose that $x \in I$ where $I$ is a nilpotent ideal of $L$ of class less than $p$. Put
\[
\hbox{exp(ad\,}x) = \sum_{r=0}^{\infty} \frac{1}{r!}(\hbox{ad\,}x)^r.
\]
Then exp(ad\,$x)$ is an automorphism of $L$. We shall call the group ${\mathcal I}(L)$ generated by all such automorphisms the group of {\em inner automorphisms} of $L$. More generally, if $B$ is a subalgebra of $L$ we denote by ${\mathcal I}(L:B)$ the group of automorphisms of $L$ generated by the exp(ad\,$x)$ with $x \in B$. Two subsets $U, V$ are {\em conjugate under ${\mathcal I}(L:B)$} if $U = \alpha(V)$ for some $\alpha \in {\mathcal I}(L:B)$; they are {\em conjugate in $L$} if they are conjugate under ${\mathcal I}(L) = {\mathcal I}(L:L)$.
\par

If $U$ is a subalgebra of $L$, the {\em centraliser} of $U$ in $L$ is the set $C_L(U) = \{ x \in L : [x,u] = 0 \}$. In \cite{cohom} Barnes showed that if $A$ is a minimal ideal of $L$ that is equal to its own centraliser in $L$, then $A$ is complemented in $L$ and all complements are conjugate under ${\mathcal I}(L:A)$. In \cite{stit} Stitzinger extended this result by finding necessary and sufficient conditions for two complements of an arbitrary minimal ideal of $L$ to be conjugate.  

\begin{theor}\label{t:bij} (\cite[Theorem 1]{stit}) Let $A$ be a minimal ideal of the solvable Lie algebra $L$. Then there is a bijection between the set ${\mathcal M}$ of conjugacy classes of complements to $A$ under ${\mathcal I}(L:A)$ and the set ${\mathcal N}$ of complements to $A$ in $C_L(A)$ that are ideals of $L$.
\end{theor}

\begin{coro}\label{c:bij} (\cite[Corollary]{stit}) Suppose that $L$ is a solvable Lie algebra and let $M, K$ be complements to a minimal ideal $A$ of $L$. Then $M$ and $K$ are conjugate under ${\mathcal I}(L:A)$ if and only if $M \cap C_L(A)$ = $K \cap C_L(A)$.
\end {coro}

Clearly, such complements are maximal subalgebras of $L$. The purpose of this paper is to consider further when two maximal subalgebras of $L$ are conjugate, and to investigate their intersection.

\begin{lemma}\label{l:mon}
Let $L$ be a solvable Lie algebra, and let $M, K$ be two core-free maximal subalgebras of $L$. Then $M, K$ are conjugate under exp(ad\,$a) = 1 + $ad\,$a$ for some $a \in A$; in particular, they are conjugate in $L$.
\end{lemma}
\begin{proof} Let $A$ be a minimal abelian ideal of $L$. Then $L = A \oplus M = A \oplus K$, $C_L(A) = A$ and $A$ is the unique minimal ideal of $L$, by \cite[Lemma 1.3]{perm}. The result, therefore, follows from \cite[Theorem 1.1]{bn}.
\end{proof}
\bigskip
 
If $U$ is a subalgebra of $L$, its {\em core}, $U_L$, is the largest ideal of $L$ contained in $U$. 
 
\begin{theor}\label{t:core} Suppose that $L$ is a solvable Lie algebra over a field $F$. If $F$ has characteristic $p$ suppose further that $L^2$ has nilpotency class less than $p$. Let $M, K$ be maximal subalgebras of $L$. Then $M$ is conjugate to $K$ in $L$ if and only if $M_L = K_L$.
\end{theor}
\begin{proof} Suppose first that $M, K$ are conjugate in $L$, so that $K = \alpha(M)$ for some $\alpha \in {\mathcal I}(L)$. Then it is easy to see that exp(ad\,$x)(M_L) = M_L$ whenever exp(ad\,$x)$ is an automorphism of $L$, whence $K_L = \alpha(M_L) = M_L$.
\par

Conversely, suppose that $M_L = K_L$. Then $M/M_L, K/M_L$ are core-free maximal subalgebras of $L/M_L$, and so are conjugate under ${\mathcal I}(L/M_L:(L/M_L)^2)$, by Lemma \ref{l:mon}. But now $M$ and $K$ are conjugate under ${\mathcal I}(L:L^2)$ by \cite[Lemma 5]{cart}, and so are conjugate in $L$.
\end{proof}
\bigskip

The above result does not hold for all solvable Lie algebras, as the following example shows.

\begin{ex}\label{e:max} Let $F$ be a field of characteristic $p$ and consider the Lie algebra $L = (\oplus_{i=0}^{p-1} Fe_i) \dot{+} Fx \dot{+} Fy$ with $[e_i,x] = e_{i+1}$ for $i = 0, \ldots, p-2$, $[e_{p-1},x] = e_0$, $[e_i,y] = ie_i$ for $i = 0, \ldots,p-1$, $[x,y] = x$, and all other products zero. Let $C$ be a faithful completely reducible $L$-module. Since $L$ is monolithic with monolith $A = \oplus_{i=0}^{p-1} Fe_i$, $C$ has a faithful irreducible submodule $B$. Let $X$ be the split extension of $B$ by $L$. Then $A + Fx + Fy$ and $A + F(x + e_1) + Fy$ are maximal subalgebras of $X$, both of which have $A$ as their core. However, $B$ is the unique minimal ideal of $L$ and these subalgebras are not conjugate under inner automorphisms of the form $1 +$ ad\,$b$, $b \in B$. Since $B$ is the nilradical of $X$, defining other inner automorphisms is problematic. 
\end{ex}

Let  $0 = L_0 < L_1 < \ldots < L_n = L$ be a chief series for $L$ and let $M$ be a maximal subalgebra of $L$. Then there exists $k$ with $0 \leq k \leq n-1$ such that $L_k \subseteq M$ but $L_{k+1} \not \subseteq M$. Clearly $L = M + L_{k+1}$ and $M \cap L_{k+1} = L_k$; we say that the chief factor $L_{k+1}/L_k$ is {\em complemented} by $M$. 

\begin{theor}\label{t:conj} Suppose that $L$ is a solvable Lie algebra over a field $F$. If $F$ has characteristic $p$ suppose further that $L^2$ has nilpotency class less than $p$. Let $A/B$ be a chief factor of $L$ that is complemented by a maximal subalgebra $M$ of $L$. If $K$ is conjugate to $M$ in $L$ then $K =$ exp(ad\,$a)(M)$ for some $a \in A$ and $M \cap K = \{m \in M:[m,a] \in M \}$.
\end{theor}
\begin{proof} We have that $L = A + M$ with $A^2 \subseteq M_L$, $M^2 \subseteq M_L$, $B \subseteq A \cap M_L$, and $M_L = K_L$. Clearly then $B \subseteq A \cap M_L \subseteq A$. Moreover, $A \neq A \cap M_L$ since $A \not \subseteq M$. It follows that $B = A \cap M_L$ because $A/B$ is a chief factor. Thus
$$ \frac{A + M_L}{M_L} \cong \frac{A}{A \cap M_L} = \frac{A}{B}, $$
whence $(A + M_L)/M_L$ is a minimal abelian ideal of $L/M_L$. Lemma \ref{l:mon} implies that $K/M_L =$ exp(ad\,$(a + M_L))(M/M_L)$ for some $a \in A$. 
\par

Now $[L,A] \subseteq B$ or $[L,A] = A$. The former implies that $[L,A] \subseteq M_L$, contradicting the fact that $(A+M_L)/M_L$ is self-centralising in $L/M_L$, by \cite[Lemma 1.4]{perm}. Hence $A = [L,A] \subseteq L^2$, and so exp(ad\,$a)$ is defined. If $x \in$ exp(ad\,$a)(M)$ then $x + M_L =$ exp(ad\,$a)(m) = m + [m,a] + M_L \in K/M_L$ for some $m \in M$, whence $x \in K$. Since exp(ad\,$a)$ is an automorphism of $L$ we must have $K =$ exp(ad\,$a)(M)$.
\par

Finally we have $(M \cap K)/M_L = (M/M_L) \cap (K/M_L) = C_{(M/M_L)}(a + M_L)$ by \cite[Lemma 1.5]{perm}. We infer that $m \in M \cap K \Leftrightarrow [m,a] \in M_L \Leftrightarrow [m,a] \in M$.
\end{proof}
\bigskip

Theorem \ref{t:conj} gave a characterisation of the intersection of two conjugate maximal subalgebras of $L$. Finally we consider the intersection of two non-conjugate maximal subalgebras of $L$.

\begin{theor}\label{t:notconj} Suppose that $L$ is a solvable Lie algebra over a field $F$. Let $M, K$ be maximal subalgebras of $L$, and suppose that $K_L \not \subseteq M_L$. Then $M \cap K$ is a maximal subalgebra of $M$. 
\end{theor}
\begin{proof} We have that $K_L \not \subseteq M$, so $L = M + K_L = M + K$. If $K = K_L$ then $L/K \cong M/(M \cap K)$ and the result is clear. So suppose that $K \neq K_L$. Let $A/K_L$ be a minimal ideal of $L/K_L$. Then $L/K_L = A/K_L \oplus K/K_L$, from \cite[Lemma 1.4]{perm}, giving $A \cap K = K_L$. Also, $A = A \cap (M + K_L) = A \cap M + K_L$, whence 
$$ \frac{A}{K_L} = \frac{A \cap M + K_L}{K_L} \cong \frac{A \cap M}{K_L \cap M}, $$  showing that $(A \cap M)/(K_L \cap M)$ is a minimal ideal of $M/(K_L \cap M) (\cong L/K_L$) and $A \cap M$ is a minmal ideal of $M$. 
Now
$$ \hbox{dim} \left( \frac{M}{M \cap K}\right) \geq \hbox{dim} \left( \frac{A \cap M + M \cap K}{M \cap K}\right) = \hbox{dim} \left( \frac{A \cap M}{K_L \cap M}\right) $$ $$= \hbox{dim} \left( \frac{A}{K_L}\right)
 = \hbox{dim} \left( \frac{L}{K}\right) = \hbox{dim} \left( \frac{M + K}{K} \right) = \hbox{dim} \left( \frac{M}{M \cap K} \right). $$
It follows that $M = A \cap M + M \cap K$ which yields the result.
\end{proof}

\begin{coro}\label{c:notconj} Suppose that $L$ is a solvable Lie algebra over a field $F$. If $F$ has characteristic $p$ suppose further that $L^2$ has nilpotency class less than $p$. Let $M, K$ be maximal subalgebras of $L$ that are not conjugate in $L$. Then $M \cap K$ is a maximal subalgebra of at least one of $M, K$.
\end{coro}

\begin{coro}\label{c:max} Suppose that $L$ is a solvable Lie algebra and let $M, K$ be complements to a minimal ideal $A$ of $L$ that are not conjugate in $L$. Then $M \cap K$ is a maximal subalgebra of both $M$ and $K$.
\end {coro}
\begin{proof} By Theorem \ref{t:bij}, both $M_L$ and $K_L$ are complements to $A$ in $C_L(A)$. Since $M_L \neq K_L$ we have $M_L \not \subseteq K_L$ and $K_L \not \subseteq M_L$. The result now follows from Theorem \ref{t:notconj}.
\end{proof}

\end{document}